\documentclass[10pt]{amsart}

\usepackage{amssymb,latexsym, mathtools,tikz}
\usepackage{enumerate}
\usepackage{graphicx}
\usepackage{float}
\usepackage{placeins}
\usepackage{mdframed}
\usepackage{amssymb}
\usepackage{esint}
\usepackage{cool}
\usepackage[all,cmtip]{xy}
\usepackage{mathtools}
\usepackage{amstext} 
\usepackage{array}   
\usepackage[shortlabels]{enumitem}
\usepackage{ytableau}

\newcolumntype{L}{>{$}l<{$}} 
\newtheorem{theorem}{Theorem}[section]
\newtheorem{lemma}[theorem]{Lemma}
\newtheorem{cor}[theorem]{Corollary}
\newtheorem{prop}[theorem]{Proposition}
\newtheorem{setup}[theorem]{Setup}
\theoremstyle{definition}
\newtheorem{definition}[theorem]{Definition}

\newtheorem{notation}[theorem]{Notation}

\theoremstyle{remark}
\newtheorem{remark}[theorem]{Remark}

\newtheorem{the context}[theorem]{The Context}
\newtheorem{question}[theorem]{Question}
\numberwithin{equation}{theorem}
\numberwithin{equation}{section}

\newcommand{\cat}[1]{\mathcal{#1}}

\newcommand{\tor}{\operatorname{Tor}}

\newcommand{\Ker}{\operatorname{Ker}}

\newcommand{\ideal}[1]{\mathfrak{#1}}
\newcommand{\m}{\ideal{m}}

\renewcommand{\geq}{\geqslant}
\renewcommand{\leq}{\leqslant}
\renewcommand{\ker}{\Ker}

\newcommand{\maps}[5]{\xymatrix{#1 \ar[r]^-{#3} & #2 \\
#4 \ar@{|->}[r] & #5 \\}}

\newcommand{\mfa}{\mathfrak{a}}

\def\w{\wedge}

\begin{document}
\title[Vanishing of Avramov Obstructions]{Vanishing of Avramov Obstructions for Products of Sequentially Transverse Ideals}



\author{Keller VandeBogert}
\date{\today}

\maketitle

\begin{abstract}
    Two ideals $I$ and $J$ are called transverse if $I \cap J = IJ$. We show that the obstructions defined by Avramov for classes of (sequentially) transverse ideals in regular local rings are always $0$. In particular, we compute an explicit free resolution and Koszul homology for all such ideals. Moreover, we construct an explicit trivial Massey operation on the associated Koszul complex and hence (by Golod's construction) a minimal free resolution of the residue field over the quotient defined by the product of transverse ideals. We conclude with questions about the existence of associative multiplicative structures on the minimal free resolution of the quotient defined by products of transverse ideals.
\end{abstract}

\section{Introduction}

Let $(R , \m , k)$ be a regular local ring. Given an ideal $I \subseteq R$, let $(F_\bullet,d_\bullet)$ denote a minimal free resolution of $R/I$. It is always possible to construct a morphism of complexes $(F \otimes_R F)_\bullet \to F_\bullet$ extending the identity in homological degree $0$; this induces a product $\cdot: F_i \otimes F_j \to F_{i+j}$. Tracing through the definition of the tensor product complex, one finds that this product satisfies the following identity:
$$d_{i+j} (f_i \cdot f_j) = d_i (f_i) \cdot f_j + (-1)^i f_i \cdot d_j (f_j).$$
In general, this product need not be associative (though it is always associative up to homotopy). When a product satisfying the above identity \emph{is} associative, we say that $F_\bullet$ admits the structure of an associative DG algebra.

The interest in constructing associative DG-algebra structures on minimal free resolutions was (arguably) sparked by Buchsbaum and Eisenbud in \cite{buchsbaum1977algebra}. In this paper, a version of what became the Buchsbaum-Eisenbud-Horrocks conjecture was proved for resolutions of length $3$; moreover, the authors conjectured the stronger result that any minimal free resolution of a quotient ring $R/I$ admits the structure of a commutative associative DG-algebra. This latter conjecture turned out to be false, with a counterexample having been produced in the context of Massey products by Khinich (see the appendix of \cite{avramov1974hopf}). The Buchsbaum-Eisenbud-Horrocks conjecture has remained open, however (although the weaker total rank conjecture was proved by Walker \cite{walker2017total} in $2017$).

In \cite{avramov1981obstructions}, Avramov constructed obstructions that could detect the non-existence of DG-module structures of complexes over other complexes. If $R/I$ has a minimal DG-algebra resolution $F_\bullet$, then for any complete intersection $\mfa \subseteq I$, $F_\bullet$ admits the structure of a DG-module over the resolution of $R/\mfa$. Using this fact, Avramov proved that the minimal free resolution $F_\bullet$ of the quotient defined by $(x_1^2 , x_1 x_2 , x_2 x_3 , x_3 x_4 , x_4^2) \subset R := k[x_1 , \dots , x_4]$ does \emph{not} admit a DG-module structure over the resolution of $R/(x_1^2 , x_4^2)$, whence $F_\bullet$ admits no associative multiplicative structure. One is tempted to ask if the vanishing of these obstructions is sufficient for the existence of such a multiplicative structure, but Srinivasan has shown that if $F_\bullet$ is the minimal free resolution of the ideal of $4 \times 4$ pfaffians of a $6 \times 6$ skew symmetric matrix, then $F_\bullet$ does not admit the structure of an associative DG-algebra, even though the obstructions defined by Avramov \emph{do} vanish (see \cite{srinivasan1992non}).

In this paper, we explore the vanishing of the aforementioned obstructions for the minimal free resolution of products of so-called sequentially transverse ideals (see Definition \ref{def:transverse}). In particular, this involves the construction of a minimal free resolution for all such products (see Proposition \ref{prop:resofTransverse}), and the construction of a DG-module structure over the minimal free resolution of any complete intersection contained in these products (see Theorem \ref{thm:Dgmodules}). Along the way, we compute Koszul homology for products of sequentially transverse ideals and ask whether or not the minimal free resolution of the quotient defined by a product of sequentially transverse ideals admits the structure of an associative DG-algebra, assuming each ideal separately has a DG-algebra minimal free resolution.

Quotient rings defined by products of transverse ideals have also been called ``minimal intersections" in work by Jorgensen and Moore (see \cite{jorgensen2009minimal}), where it is shown that certain classes of modules over minimal intersections have nontrivial vanishing of Ext and Tor. Likewise, Avramov has studied the relationship between the homological data of the rings $R/I$, $R/J$, and $R/I \otimes_R R/J$ for transverse ideals $I$ and $J$ when $R$ is a regular local ring (see \cite{avramov1975homology}).

The paper is organized as follows. In Section \ref{sec:background}, we introduce the conventions, notation, and definitions to be used throughout the rest of the paper, including the definition of the Avramov obstructions (see Definiton \ref{def:avramovobstructions}). In Section \ref{sec:theResolution}, we show that the free resolution of the product of transverse ideals may be computed explicitly based on the resolutions of each ideal separately. The complex constructed here has been called the \emph{star product} by Geller in \cite{gellerFiberProds}. Next, we compute the Koszul homology for products of sequentially transverse ideals and use this to produce an explicit minimal free resolution of the residue field $k$ over the ring $R/IJ$. This particular resolution is much simpler than the general construction of Golod, due to the trivial Massey operation computed in Lemma \ref{cor:trivialprod}.

Finally, in Section \ref{sec:VanishingofObs}, we show that (under suitable hypotheses) the Avramov obstructions for products of sequentially transverse ideals are trivial. In particular, combining this with the triviality of the Tor-algebra multiplication, we obtain the injectivity of certain maps of Tor (see Corollary \ref{cor:injTor}). Finally, we end with a question about the existence of multiplicative structures on the complex of Definition \ref{def:starprod}.

\section{Transverse Ideals and Avramov Obstructions}\label{sec:background}

In this section, we introduce some necessary background, not least of which is the definition of the \emph{Avramov obstructions} to the existence of multiplicative structures on resolutions. The connection between these obstructions and the existence of DG-module (and hence DG-algebra) structures is recalled. Next, the main objects of study in this paper, \emph{transverse} ideals, are defined along with some standard results on Tor-rigidity.

Throughout the paper, all complexes will be assumed to have nontrivial terms appearing only in nonnegative homological degrees. 

\begin{notation}
The notation $(F_\bullet , d_\bullet)$ will denote a complex $F_\bullet$ with differentials $d_\bullet$. When no confusion may occur, $F$ may be written, where the notation $d^F$ is understood to mean the differential of $F$ (in the appropriate homological degree).

Given a complex $F_\bullet$ as above, elements of $F_n$ will often be denoted $f_n$, without specifying that $f_n \in F_n$.
\end{notation}

\begin{definition}\label{def:dga}
A \emph{differential graded algebra} $(F,d)$ (DG-algebra) over a commutative Noetherian ring $R$ is a complex of finitely generated free $R$-modules with differential $d$ and with a unitary, associative multiplication $F \otimes_R F \to F$ satisfying
\begin{enumerate}[(a)]
    \item $F_i F_j \subseteq F_{i+j}$,
    \item $d_{i+j} (x_i x_j) = d_i (x_i) x_j + (-1)^i x_i d_j (x_j)$,
    \item $x_i x_j = (-1)^{ij} x_j x_i$, and
    \item $x_i^2 = 0$ if $i$ is odd,
\end{enumerate}
where $x_k \in F_k$.
\end{definition}

\begin{prop}\label{prop:dgmoduleoverK}
Let $(R , \m , k)$ denote a regular local ring. Suppose the minimal free resolution $F_\bullet$ of $R/I$ admits the structure of an associative DG-algebra. Then for any complete intersection $\mfa \subseteq I$, $F_\bullet$ admits the structure of a DG $K_\bullet$-module, where $K_\bullet$ is the minimal free resolution of $R/\mfa$.
\end{prop}
The proof of Proposition \ref{prop:dgmoduleoverK} is essentially that of Proposition \ref{prop:extendproduct}.

\begin{definition}\label{def:avramovobstructions}
Let $(R, \m ,k )$ denote a local ring and $f : R \to S$ a morphism of rings. Let $\tor_+^R (S , k)$ denote the subalgebra of $\tor_\bullet^R (S , k)$ generated in positive homological degree. For any $S$-module $M$, there exists a map of graded vector spaces:
$$\frac{\tor_\bullet^R (M , k)}{\tor_+^R (S,k) \cdot \tor_\bullet^R (M , k)} \to \tor_\bullet^S (M , k).$$
The kernel of this map is denoted $o^f (M)$ and is called the \emph{Avramov obstruction}.
\end{definition}
The following Theorem makes clear why $o^f (M)$ is referred to as an obstruction.

\begin{theorem}[\cite{avramov1981obstructions}, Theorem 1.2]\label{thm:obstructions}
Let $(R, \m ,k )$ denote a local ring and $f : R \to S$ a morphism of rings. Assume that the minimal $R$-free resolution $F_\bullet$ of $S$ admits the structure of an associative DG-algebra. If $o^f (M) \neq 0$, then no DG $F_\bullet$-module structure exists on the minimal $R$-free resolution of $M$.
\end{theorem}

In the case that $S = R/J$ and $J$ is generated by a regular sequence, the Avramov obstructions break into ``graded" pieces:
$$o_i^f (M) := \ker \Big( \frac{\tor_i^R (M,k)}{\tor_1^R (S,k) \cdot \tor_{i-1}^R (M , k)} \to \tor_i^S (M , k) \Big).$$
Observe that Proposition \ref{prop:dgmoduleoverK} may then be rephrased in the following manner:
\begin{prop}\label{prop:vanishofObs}
Let $(R , \m , k)$ denote a regular local ring. Suppose the minimal free resolution $F_\bullet$ of $R/I$ admits the structure of an associative DG-algebra. Then for any complete intersection $\mfa \subseteq I$, $o_i^f (R/I) = 0$ for all $i >0$, where $f : R \to R/\mfa$ is the natural quotient map.
\end{prop}
Proposition \ref{prop:vanishofObs} may tempt one to ask whether the vanishing of Avramov obstructions implies the existence of a DG algebra structure on the minimal free resolution of $R/I$. An answer in the negative was provided by Srinivasan (see \cite{srinivasan1992non}) The following definition introduces the main objects of interest in this paper.

\begin{definition}\label{def:transverse}
Two ideals $I$ and $J$ are \emph{transverse} if $I \cap J = IJ$; equivalently, if $\tor_1^R (R/I , R/J) = 0$.

A family of ideals $I_1 , \dotsc , I_n \subset R$ is \emph{sequentially transverse} if for all $i = 1 , \dots , n-1$,
$$I_1 \cdots I_i \cap I_{i+1} = I_1 \dotsc I_{i+1}.$$
Equivalently,
$$\tor_1 (R/I_1 \cdots I_i , R/I_{i+1}) = 0 \ \textrm{for all} \ i=1, \dots , n-1.$$
\end{definition}

\begin{definition}\label{def:torindependent}
Let $(R,\m)$ be a local ring. Two $R$-modules $M$ and $N$ are \emph{Tor-independent} if $\tor_{i}^R (M, N) = 0$ for all $i \geq 1$. 
\end{definition}
In the case that $(R , \m ,k)$ is a regular local ring, the property of being transverse implies the (generally stronger) property of Tor-independence for the relevant quotient rings.

\begin{prop}\label{prop:rigidity}
Let $(R, \m)$ be a regular local ring. If $I$ and $J$ are transverse ideals, then $R/I$ and $R/J$ are Tor-independent.
\end{prop}

\begin{proof}
This is just a restatement of the well-known rigidity of Tor over regular local rings (see \cite[Corollary 2.5]{celikbas2015vanishing}).
\end{proof}

\section{Minimal Free Resolution of Products of Transverse Ideals}\label{sec:theResolution}

In this section, our goal is to give an explicit (minimal) free resolution of the quotient defined by products of transverse ideals. The construction used to do this involves a so-called \emph{star-product}, introduced independently by Geller \cite{gellerFiberProds}. 
It is worth noting that throughout this section, the condition that $(R, \m ,k )$ is a regular local ring can be relaxed to the condition that $R$ is an arbitrary commutative ring with $R/I$ and $R/J$ Tor-independent. 

The following is often (lovingly) called the ``stupid" truncation (see, for instance, \cite[Section 12.15]{stacks-project}).
\begin{notation}
Let $(F_\bullet,d_\bullet)$ denote a complex. The notation $(F_{\geq n}, d_{\geq n} )$ will denote the complex with
$$\big( F_{\geq n} \big)_i := \begin{cases}
F_i & \textrm{if} \ i \geq n \\
0 & \textrm{if} \ i < n \\
\end{cases}$$
$$\big( d_{\geq n} \big)_i := \begin{cases}
d_i & \textrm{if} \ i > n \\
0 & \textrm{if} \ i \leq n \\
\end{cases}$$
\end{notation}

The following definition (and notation) was introduced first by Geller in \cite{gellerFiberProds}, originally in the context of resolving fiber products of residue fields.
\begin{definition}\label{def:starprod}
Let $(F_\bullet , d^F_\bullet)$ and $(G_\bullet, d^G_\bullet)$ denote two complexes. The \emph{star-product} $(F*G)_\bullet$ is defined to be the complex induced by the differentials
$$d^{F*G}_n := \begin{cases}
d^F_1 \otimes d^G_1 & \textrm{if} \ n=1 \\
d^{F_{\geq 1} \otimes G_{\geq 1} }_{n+1} & \textrm{otherwise} \\
\end{cases}$$
\end{definition}

Observe that $(F*G)_\bullet$ is indeed a complex. It is clear that the differentials compose to $0$ for homological degrees at least $2$. For the front two differentials,
\begingroup\allowdisplaybreaks
\begin{align*}
    d_1^{F*G} \circ d_2^{F*G} (f_1 \otimes g_2 + f_2 \otimes g_1) &= d_1^{F*G} ( -f_1 \otimes d_2^G (g_2) + d_2^F (f_2) \otimes g_1 ) \\
    &= -d_1^F (f_1) \otimes d_1^G \circ d_2^G (g_2) + d_1^F \circ d_2^F (f_2) \otimes d_1^G (g_1) \\
    &= 0 .\\
\end{align*}
\endgroup
The following Proposition illustrates the relationship between the star product of Definition \ref{def:starprod} and sequentially transverse families of ideals.
\begin{prop}\label{prop:resofTransverse}
Let $(R, \m , k)$ denote a regular local ring. Let $I$ and $J$ be transverse ideals and let $F_\bullet$, $G_\bullet$ denote free resolutions of $R/I$ and $R/J$, respectively. Then $(F * G)_\bullet$ is a free resolution of $R/IJ$. If $F_\bullet$ and $G_\bullet$ are minimal, then $(F * G)_\bullet$ is also minimal.
\end{prop}

\begin{proof}
The latter statement about minimality is clear. It remains only to show acyclicity. There is a short exact sequence of complexes
$$0 \to F_{\geq 1} \oplus G_{\geq 1} \to (F \otimes G)_{\geq 1} \to (F* G)_{\geq 1}[-1] \to 0,$$
so by the long exact sequence of homology, $H_i (F*G) = 0$ for $i \geq 2$. For $i=1$, there is a short exact sequence
$$0 \to H_1 (F*G_{\geq 1} ) \to H_1 (F_{\geq 1}) \oplus H_1 (G_{\geq 1}) \to H_1 (F \otimes G_{\geq 1}) \to 0,$$
and the connecting morphism is computed as the map 
$$[f_1 \otimes g_1] \mapsto (-d^G(g_1) [f_1] , d^F(f_1) [g_1])$$ 
(where $[-]$ denotes homology class). Identifying the homology appearing in the latter two terms of the short exact sequence and using that $I \cap J = IJ$, this implies that $H_1 (F*G_{\geq 1}) \cong IJ$ via the map $[f_1 \otimes g_1] \mapsto d^F(f_1) d^G(g_1)$. Thus, augmenting $(F*G)_{\geq 1}$ by the map $d_1^F \otimes d^G_1$ remains acyclic.
\end{proof}

\begin{remark}
By iterating Proposition \ref{prop:resofTransverse}, one finds that if $I_1 , \dots , I_n$ is a family of sequentially transverse ideals with $F_\bullet^i$ a free resolution of $R/I_i$ for each $1 \leq i \leq n$, then a free resolution of $R/I_1 \cdots I_n$ is obtained as
$$(F^1 * F^2 * \cdots * F^n)_\bullet.$$
\end{remark}

\section{Koszul Homology for Products of Transverse Ideals and A Minimal Free Resolution of the Residue Field}\label{sec:Koszulhom}

In this section, we compute Koszul homology for quotients defined by products of sequentially transverse ideals. The basis given can be thought of a ``shifted" version of a K\"unneth formula. Moreover, we use this explicit Koszul homology to compute a trivial Massey operation on the Koszul homology algebra of $R/IJ$; this allows us to employ a standard construction of Golod to deduce a rather simple minimal free resolution of the residue field $k$ over the ring $R/IJ$. 

\begin{definition}
Let $(R, \m ,k)$ denote a local ring or a standard graded polynomial ring over a field $k$. The Koszul homology of an $R$-module $M$, denoted $H_i (M)$, is defined as
$$H_i (M) := H_i (K_\bullet \otimes_R M),$$
where $K_\bullet$ is the Koszul complex resolving $k$.
\end{definition}

\begin{theorem}\label{thm:isoofKoszul}
Let $I$ and $J$ be transverse ideals. Then the morphism of vector spaces
\begingroup\allowdisplaybreaks
\begin{align*}
    \bigoplus_{\substack{i+j=n+1 \\ 
    i, \ j \geq 1}} H_i (R/I) \otimes_k H_j (R/J) &\to H_n (R/IJ) \\
    [z_1^I] \otimes [z_2^J] &\mapsto [z_1^I \w d(z_2^J) ] \\
\end{align*}
\endgroup
is an isomorphism.
\end{theorem}

\begin{proof}
We will instead compute the Koszul homology $H_i (IJ)$ and use the isomorphism $H_{i+1} (R/IJ) \cong H_i (IJ)$ given by sending a cycle $[z] \mapsto [d(z)]$ (where $d$ denotes the Koszul differential). Using a K\"unneth formula along with the isomorphism just mentioned, $H_n (I+J)$ has basis given by all elements of the form
$$[d(z_1) \w z_2 ] + (-1)^{|z_1|} [ z_1 \w d(z_2) ],$$
where $z_1$ and $z_2$ represent basis elements for $H_i (I)$ and $H_{n-i} (J)$, respectively (for some $0 \leq i \leq n$). Using these basis elements along with the short exact sequence
$$0 \to IJ \to I \oplus J \to I+J \to 0,$$
one immediately computes the connecting homomorphism $H_n (I+J ) \to H_{n-1} (IJ)$ as
$$[d(z_1 \w z_2 )] \mapsto \begin{cases} (-1)^{|z_1|} [ d(z_1) \w d(z_2)] & \textrm{if} \ 1 \leq |z_1| \leq n-1 \\
0 & \textrm{otherwise}. \\
\end{cases}$$
Moreover, since $H_n(I ) \oplus H_n (J) \to H_n (I+J)$ is an injection, the map $H_n (IJ) \to H_n(I) \oplus H_n (J)$ must be the $0$ map. This means that the long exact sequence of homology splits into short exact sequences
$$0 \to H_n (I) \oplus H_n(J) \to H_n (I+J) \to H_{n-1} (IJ) \to 0.$$
This immediately implies that $H_{n-1} (IJ)$ has basis given by all cycles of the form $[d(z_1) \w d(z_2)]$, where $[z_1]$ and $[z_2]$ are basis elements of $H_i (I)$ and $H_{n-i} (I)$, respectively, and $1 \leq i \leq n-1$. To conclude the proof, simply observe that
$$\big[ d(z_1 \w d(z_2) ) \big] = [d(z_1) \w d(z_2)].$$
\end{proof}

\begin{remark}
If $I_1 , \dots , I_n$ is a family of sequentially transverse ideals, then an iteration of Theorem \ref{thm:isoofKoszul} yields that $H_\ell (R/I_1 \cdots I_n)$ has basis given by all elements if the form
$$[z_1^{I_1} \w d(z_2^{I_2} ) \w \cdots \w d(z_n^{I_n}) ],$$
where each $[z_i^{I_i}]$ is a basis element for $H_{j_i} (R/I_i)$, and $j_1 + \cdots + j_n = \ell + n-1$, $j_i \geq 1$ for each $i$.

Moreover, observe that the Leibniz rule implies that $[z_1 \w d(z_2)]$ is a basis if and only if $[d(z_1) \w z_2]$ is also a basis for the Koszul homology, implying that Theorem \ref{thm:isoofKoszul} is symmetric with respect to $I$ and $J$. 
\end{remark}

For the remainder of this section, our goal will be to construct the minimal free resolution of the residue field $k$ over the quotient ring $R/IJ$. This will involve the construction of a trivial Massey operation on the algebra $R/IJ \otimes_R K_\bullet$ for transverse ideals $I$ and $J$; for the definition of trivial Massey operations and the result employed in Corollary \ref{cor:resofk}, see Avramov's well-known work \cite[Section 5]{avramov1998infinite}.

\begin{setup}\label{set:residueFieldSet}
Let $I$ and $J$ be transverse ideals, and enumerate $k$-bases $\{ [z_{a}^I] \}_{a \in A}$ and $\{ [z_b^J] \}_{b \in B}$ for the homology algebras $H_{\geq 1} (R/I)$ and $H_{\geq 1} (R/J)$, respectively. By Theorem \ref{thm:isoofKoszul}, a $k$-basis for the homology algebra $H_{\geq 1} (R/ IJ)$ is given by $\cat{B} := \{ [d(z_a^I) \w z_b^J ] \}_{a \in A, \ b\in B}$. 

For any $\ell \geq 1$, let $V_\ell$ denote the free $R$-module with formal basis elements
$$\{ v_{a,b} \mid a \in A, \ b \in B, \ \textrm{and} \ \ell = |z_a^I| + |z_b^J|  \}.$$
For convenience and conciseness, use the notation $h_{a,b} := [d(z_a^I) \w z_b^J] )$ for $a \in A$, $b \in B$.
\end{setup}

\begin{lemma}\label{cor:trivialprod}
Adopt notation and hypotheses as in Setup \ref{set:residueFieldSet}. Let 
\begingroup\allowdisplaybreaks
\begin{align*}
    \mu : \coprod_{i =1}^\infty \cat{B}^i &\to H_{\geq 1} (R/IJ), \\
    \mu([d(z_{a_1}^I) \w z_{b_1}^J ], \cdots , [d(z_{a_p}^I) \w z_{b_p}^J ] ) &= d(z_{a_1}^I ) \w z_{b_1}^J \w z_{a_2}^I \w z_{b_2}^J \w \cdots \w z_{a_p}^I \w z_{b_p}^J. \\
\end{align*}
\endgroup
Then $\mu$ is a trivial Massey operation on $R/IJ \otimes_R K_\bullet$.
\end{lemma}

\begin{proof}
Throughout the proof, let $\overline{z} := (-1)^{|z|+1} z$. By definition, $\mu$ is well defined on singleton elements; inductively, one computes:
\begingroup\allowdisplaybreaks
\begin{align*}
    &\sum_{i=1}^{p-1} \overline{\mu (h_{a_1,b_1} , \dots , h_{a_i,b_i})} \mu (h_{a_{i+1}, b_{i+1}} , \dots , h_{a_p,b_p} ) \\
    =& \sum_{i=1}^{p-1} (-1)^{1+\sum_{j=1}^i |z_{a_j}^I \w z_{b_j}^J|} d(z_{a_1}^I ) \w z_{b_1}^J \w \cdots \w d(z_{a_{i+1}}^I) \w z_{b_{i+1}}^J  \w \cdots \w z_{a_p} \w z_{b_p} \\
    =& \sum_{i=1}^{p-1} (-1)^{1+\sum_{j=1}^i |z_{a_j}^I \w z_{b_j}^J|} d(z_{a_1}^I ) \w z_{b_1}^J \w \cdots \w d(z_{a_{i+1}}^I \w z_{b_{i+1}}^J)  \w \cdots \w z_{a_p} \w z_{b_p} \\
    = & (-1)^{|z_{a_1}^I \w z_{b_1}^J|+1} d(z_{a_1}^I) \w z_{b_1}^J \w d \big( z_{a_2}^I \w z_{b_2}^J \w \cdots \w z_{a_p}^I \w z_{b_p}^J \big) \\
    = & d \Big( d(z_{a_1}^I ) \w z_{b_1}^J \w z_{a_2}^I \w z_{b_2}^J \w \cdots \w z_{a_p}^I \w z_{b_p}^J \Big)  \\
    =& d \mu (h_{a_1,b_1} , \dots , h_{a_p,b_p} ). \\
\end{align*}
\endgroup
Thus, $\mu$ is a trivial Massey operation by definition. 
\end{proof}

With explicit generators of the Koszul homology $H_\bullet (R/IJ)$ and a quite simple trivial Massey operation on the algebra $R/IJ \otimes_R K_\bullet$, we can use a construction of Golod to produce an explicit minimal free resolution of the residue field $k$ over the ring $R/IJ$.

\begin{cor}\label{cor:resofk}
Adopt notation and hypotheses as in Setup \ref{set:residueFieldSet} and let $\mu$ be the trivial Massey operation of Lemma \ref{cor:trivialprod}. Let $(T_\bullet, \partial_\bullet)$ denote the complex with
\begingroup\allowdisplaybreaks
\begin{align*}
    T_n &:= \bigoplus_{p+h+i_1 + \cdots + i_p = n} K_h \otimes_R V_{i_1} \otimes_R \cdots \otimes_R V_{i_p}, \\
    \partial_n &: T_n \to T_{n-1}, \\
    \partial_n (a \otimes v_{a_1,b_1} \otimes \cdots \otimes v_{a_p,b_p} ) &= d(a) \otimes v_{a_1,b_1} \otimes \cdots \otimes v_{a_p,b_p} \\
    + (-1)^{|a|} \sum_{j=1}^p & a \mu(h_{a_{1},b_1}, \dots , h_{a_j,b_j} ) \otimes v_{a_{j+1},b_{j+1}} \otimes \cdots \otimes v_{a_p,b_p}. \\
\end{align*}
\endgroup
Then $T_\bullet$ is the minimal free resolution of the residue field $k$ over $R/IJ$. In particular, $R/IJ$ is a Golod ring. 
\end{cor}

\begin{remark}
 Work of Herzog and Steurich (see \cite{herzog1979golodideale}) has already shown that $R/IJ$ is a Golod ring for $I$ and $J$ transverse. However, their work did not give an explicit trivial Massey operation nor the minimal free resolution of the residue field. It is worth noting that in general, not all products of ideals are Golod (see, for example, \cite[Example 2.1]{de2016products}).
\end{remark}

\begin{remark}
In general, a Golod ring may admit trivial Massey operation that is quite complicated, making the construction of the minimal free resolution of $k$ much more difficult to formulate. Another class of ideals for which the Koszul complex admits a trivial Massey operation that can be written explicitly is given by strongly stable ideals; see work of Peeva \cite[Corollary 4.2]{peeva19960}.
\end{remark}

\section{Vanishing of Obstructions}\label{sec:VanishingofObs}

In this section, we show that the Avramov obstructions of Definition \ref{def:avramovobstructions} are trivial for quotients defined by products of sequentially transverse ideals as in Setup \ref{setup:obstructionsetup}. In particular, this implies that transverse ideal with DG-algebra resolutions are such that the corresponding Avramov obstructions for the product vanish. This leads us to ask whether the complex of Definition \ref{def:starprod} admits the structure of an associative DG-algebra whenever the complexes $F$ and $G$ do. We conclude with some comments on the cases for which an algebra structure does exist.

\begin{prop}\label{prop:extendproduct}
Let $(F_\bullet, d_\bullet)$ denote a free resolution of $R/I$ for some ideal $I \subset R$. Assume that there exists a product $\cdot : F_1 \otimes F_i \to F_{i+1}$ satisfying
\begin{enumerate}[(a)]
    \item $d_{i+1} (f_1 \cdot f_i ) = d_1 (f_1) f_i - f_1 \cdot d_i (f_i)$, and
    \item $f_1 \cdot (f_1 \cdot f_i) = 0$.
\end{enumerate}
Then for any complete intersection $\mfa \subseteq I$, $F_\bullet$ is a DG-module over the Koszul complex $K_\bullet$ resolving $R/\mfa$.
\end{prop}

\begin{proof}
Let $\phi : K_\bullet \to F_\bullet$ denote any comparison map extending the identity in homological degree $0$. Define a module action $* : K_1 \otimes F_i \to F_{i+1}$ by 
$$k_1 * f_i := \phi_1 (k_1) \cdot f_i \in F_{i+1}.$$
This extends to a module action of $T(K_1)$ (the tensor algebra) on $F_\bullet$; moreover, since $k_1* (k_1 * f_i) = 0$, this action factors through the exterior algebra $\bigwedge^\bullet K_1 = K_\bullet$. 
\end{proof}

\begin{remark}
Observe that if $I$ is as in the statement of Proposition and $\mfa \subseteq I$ is any complete intersection, then the Avramov obstruction $o^f (R/I)$ vanishes, where $f : R \to R/\mfa$, by Theorem \ref{thm:obstructions}. 
\end{remark}

\begin{setup}\label{setup:obstructionsetup}
Let $I_1 , \dots , I_n$ denote a family of sequentially transverse ideals. Let $(F_\bullet^i, d_\bullet^i)$ denote a free resolution of $R/I_i$ for each $i=1 , \dots , n$. Assume that for each $i$ and $j$, there exists a product $\cdot : F_1^i \otimes F_j^i \to F_{j+1}^i$ satisfying
\begin{enumerate}[(a)]
    \item $d_{j+1}^i (f_1^i \cdot f_j^i ) = d_1^i (f_1^i) f_j^i - f_1^i \cdot d_j (f_j^i)$, and
    \item $f_1^i \cdot (f_1^i \cdot f_j^i) = 0$.
\end{enumerate}
\end{setup}

\begin{remark}
If $I_1 , \dots , I_n$ is a family of sequentially transverse ideals, each of which admits a minimal free resolution with the structure of an associative DG algebra, then the hypotheses of Setup \ref{setup:obstructionsetup} are satisfied by Definition \ref{def:dga}.
\end{remark}

\begin{theorem}\label{thm:Dgmodules}
Adopt notation and hypotheses as in Setup \ref{setup:obstructionsetup}. Let $S_\bullet := (F_\bullet^1 * \cdots * F_\bullet^n)_\bullet$. Then there exists a product
$$S_1 \otimes S_j \to S_{j+1}$$
satisfying
\begin{enumerate}[(a)]
    \item $d_{j+1}^i (f_1^i \cdot f_j^i ) = d_1^i (f_1^i) f_j^i - f_1^i \cdot d_j (f_j^i)$, and
    \item $f_1^i \cdot (f_1^i \cdot f_j^i) = 0$.
\end{enumerate}
\end{theorem}

\begin{proof}
By induction it suffices to consider the case that $I$ and $J$ are transverse ideals with $F_\bullet$, $G_\bullet$ free resolutions of $R/I$ and $R/J$, respectively. For ease of notation, elements of $F_i$ will denoted denoted $f_i$ and elements of $G_j$ will be denoted $g_j$. Define
$$(f_1 \otimes g_1) \cdot (f_a \otimes g_b ) := \begin{cases}
(-1)^a d^F_1 (f_1) f_a \otimes g_1 \cdot g_b  & \textrm{if} \ b>1 \\
(-1)^a d^F_1 (f_1) f_a \otimes g_1 \cdot g_b + d_1^G(g_b) f_1 \cdot f_a \otimes g_1 & \textrm{if} \ b=1. \\
\end{cases}$$
The cycle condition will be verified directly in the relevant cases; for ease of notation in the computations, the homological degree of the relevant differential will be suppressed. Let $S_\bullet := (F * G)_\bullet$.

\textbf{Case 1:} $b >2$. 
\begingroup\allowdisplaybreaks
\begin{align*}
    d^S ( (f_1 \otimes g_1) \cdot (f_a \otimes g_b) ) &= (-1)^a d^F (f_1) d^F (f_a) \otimes g_1 \cdot g_b +  d^F (f_1) f_a \otimes d^G (g_1 \cdot g_b) \\
    &= (-1)^a d^F (f_1) d^F (f_a) \otimes g_1 \cdot g_b + d^F (f_1) d^G (g_1) f_a \otimes g_b \\
    & - d^F (f_1) f_a \otimes g_1 \cdot d^G (g_b) \\
    &= d^S (f_1 \otimes g_1) f_a \otimes g_b - (f_1 \otimes g_1) \cdot d^S (f_a \otimes g_b) \\
\end{align*}
\endgroup
\textbf{Case 2:} $b=2$.
\begingroup\allowdisplaybreaks
\begin{align*}
    d^S ( (f_1 \otimes g_1) \cdot (f_a \otimes g_b) ) &=  (-1)^a d^F (f_1) d^F (f_a) \otimes g_1 \cdot g_b +  d^F (f_1) f_a \otimes d^G (g_1 \cdot g_b) \\
    &= (-1)^a d^F (f_1) d^F (f_a) \otimes g_1 \cdot g_b + d^F (f_1) d^G (g_1) f_a \otimes g_b \\
    & - d^F (f_1) f_a \otimes g_1 \cdot d^G (g_b) - (-1)^a d^G (d^G(g_b)) f_1 \cdot f_a \otimes g_1 \\
    &=d^S (f_1 \otimes g_1) f_a \otimes g_b - (f_1 \otimes g_1) \cdot d^S (f_a \otimes g_b) \\
\end{align*}
\endgroup
\textbf{Case 3:} $b=1$.
\begingroup\allowdisplaybreaks
\begin{align*}
    d^S ( (f_1 \otimes g_1) \cdot (f_a \otimes g_b) ) &=  (-1)^a d^F (f_1) d^F (f_a) \otimes g_1 \cdot g_b +  d^F (f_1) f_a \otimes d^G (g_1 \cdot g_b) \\
    &+ d^G (g_b) d^F(f_1) f_a \otimes g_1 - d^G(g_b) f_1 \cdot d^F (f_a) \otimes g_1 \\
    &= (-1)^a d^F (f_1) d^F (f_a) \otimes g_1 \cdot g_b + d^F (f_1) d^G (g_1) f_a \otimes g_b \\
    &- d^G (g_b) f_1 \cdot d^F (f_a) \otimes g_1 \\
    &= d^S (f_1 \otimes g_1) f_a \otimes g_b \\
    &- \big( (-1)^{a-1} d^F (f_1) d^F (f_a) \otimes g_1 \cdot g_b + d^G (g_b) f_1 \cdot d^F (f_a) \otimes g_1 \big) \\
    &= d^S (f_1 \otimes g_1) f_a \otimes g_b - (f_1 \otimes g_1) \cdot d^S (f_a \otimes g_b) \\
\end{align*}
\endgroup
It remains to show that $(f_1 \otimes g_1) \cdot \big( (f_1 \otimes g_1) \cdot (f_a \otimes g_b) \big) = 0$. This splits into two cases:

\textbf{Case 1:} $b>1$. 
\begingroup\allowdisplaybreaks
\begin{align*}
    (f_1 \otimes g_1) \cdot \big( (f_1 \otimes g_1) \cdot (f_a \otimes g_b) \big) &= (f_1 \otimes g_1) \cdot ((-1)^a d^F (f_1) f_a \otimes g_1 \cdot g_b \\
    &= d^F (f_1)^2 f_a \otimes g_1 \cdot (g_1 \cdot g_b) \\
    &= 0 \\
\end{align*}
\endgroup
\textbf{Case 2:} $b=1$.
\begingroup\allowdisplaybreaks
\begin{align*}
     (f_1 \otimes g_1) \cdot \big( (f_1 \otimes g_1) \cdot (f_a \otimes g_b) \big) &= (f_1 \otimes g_1) \cdot \big( (-1)^a d^F_1 (f_1) f_a \otimes g_1 \cdot g_b + d_1^G(g_b) f_1 \cdot f_a \otimes g_1 \big) \\
     &= d^F (f_1)^2 f_a \otimes g_1 \cdot (g_1 \cdot g_b) + (-1)^{a+1}d^G (g_b) d^F (f_1) f_1 \cdot f_a \otimes g_1 \cdot g_1 \\
     &+ d^G (g_b) d^G (g_1) f_1 \cdot (f_1 \cdot f_b) \otimes g_1 \\
     &= 0 .\\
\end{align*}
\endgroup
\end{proof}

\begin{cor}\label{cor:injTor}
Adopt notation and hypotheses as in Setup \ref{setup:obstructionsetup} and define $J := I_1 \cdots I_n$. If $\mfa \subset J$ is any complete intersection, then the induced map
$$\tor_i^R (R/J,k) \to \tor^S_i (R/J , k)$$
is injective for all $i \geq 2$, where $S = R/\mfa$. 
\end{cor}

\begin{proof}
Observe that by triviality of the multiplication in the Tor algebra $\tor_\bullet^R (R/J , k)$, one has
$$\tor_1^R (S , k ) \cdot \tor_{i-1} (R/J , k) = 0$$
for all $i \geq 2$. The statement of the corollary is then a rephrasing of the fact that $o_i^f (R/J) = 0$, where $f : R \to S$ is the natural quotient.
\end{proof}

\begin{question}\label{question:isitDG}
Let $(R , \m , k)$ be a regular local ring with $I$ and $J$ transverse ideals in $R$. Assume that $R/I$ and $R/J$ have minimal free resolutions admitting the structure of an associative DG algebra. Then, does the minimal free resolution of $R/IJ$ admit the structure of an associative DG algebra?
\end{question}

An answer to Question \ref{question:isitDG} in either the positive or negative is interesting. If the answer is positive, then all of the desirable properties of rings with minimal algebra resolutions hold for $R/IJ$. In the negative case, however, this would be another class of rings with the property that the associated Avramov obstructions vanish, even though there does \emph{not} exist an algebra structure on the minimal free resolution. To the author's knowledge, the only other known examples of such rings were given by Srinivasan in \cite{srinivasan1992non}.

In the case that $(F*G)_\bullet$ has length $ \leq 3$, the product of Theorem \ref{thm:obstructions} \emph{does} yield an associative algebra structure (though it is already well known by work of Buchsbaum and Eisenbud in \cite{buchsbaum1977algebra} that an associative algebra structure exists). Moreover, it is known that if $F$ and $G$ are either Koszul or Taylor complexes, then $F*G$ may be endowed with the structure of an associative DG algebra by work of Geller. Thus, any counterexample to Question \ref{question:isitDG} must come from a product of transverse ideals such that at least one of the ideals is \emph{not} a complete intersection.

\section*{Acknowledgements}

Thanks to Andy Kustin and Josh Pollitz for helpful comments/corrections on earlier drafts of this paper.

\bibliographystyle{amsplain}
\bibliography{biblio}
\addcontentsline{toc}{section}{Bibliography}

\end{document}